\documentclass[11pt,letterpaper]{article}

\usepackage[margin=1in]{geometry}
\usepackage[affil-it]{authblk}
\usepackage{amsmath,amssymb,amsthm,mathtools}
\usepackage{bbm}
\usepackage{enumitem}
\usepackage{microtype}
\usepackage{xcolor}
\usepackage{aliascnt}
\usepackage[colorlinks=true,linkcolor=blue!55!black,citecolor=green!45!black,urlcolor=blue!65!black]{hyperref}
\usepackage[capitalize,noabbrev]{cleveref}

\usepackage{tikz}
\usetikzlibrary{positioning,arrows.meta}

\allowdisplaybreaks
\setlist[itemize]{leftmargin=1.6em,itemsep=0.25em,topsep=0.4em}
\setlist[enumerate]{leftmargin=1.8em,itemsep=0.25em,topsep=0.4em}

\newtheorem{theorem}{Theorem}[section]
\newaliascnt{proposition}{theorem}
\newtheorem{proposition}[proposition]{Proposition}
\aliascntresetthe{proposition}
\newaliascnt{lemma}{theorem}
\newtheorem{lemma}[lemma]{Lemma}
\aliascntresetthe{lemma}
\newaliascnt{corollary}{theorem}
\newtheorem{corollary}[corollary]{Corollary}
\aliascntresetthe{corollary}
\newaliascnt{claim}{theorem}

\aliascntresetthe{claim}
\theoremstyle{definition}
\newaliascnt{definition}{theorem}

\aliascntresetthe{definition}
\theoremstyle{remark}
\newaliascnt{remark}{theorem}

\aliascntresetthe{remark}

\newcommand{\R}{\mathbb{R}}
\newcommand{\E}{\mathbb{E}}
\newcommand{\Pp}{\mathbb{P}}
\newcommand{\1}{\mathbbm{1}}
\newcommand{\op}{\mathrm{op}}

\newcommand{\KL}{\mathrm{D}_{\mathrm{KL}}}
\newcommand{\Cov}{\operatorname{Cov}}
\newcommand{\Var}{\operatorname{Var}}
\newcommand{\Tr}{\operatorname{Tr}}
\newcommand{\diag}{\operatorname{diag}}

\newcommand{\norm}[1]{\lVert #1 \rVert}
\newcommand{\dd}{\mathrm{d}}

\title{Optimal Covariance Inflation under Gaussian Tilts}
\author[1,2]{Minbo Gao\thanks{\href{mailto:gmb17@tsinghua.org.cn}
  {\texttt{gmb17@tsinghua.org.cn}}}}
\author[3]{Zhengfeng Ji\thanks{\href{mailto:jizhengfeng@tsinghua.edu.cn}
  {\texttt{jizhengfeng@tsinghua.edu.cn}}}}
\author[1,2]{Chenghua Liu\thanks{\href{mailto:liuch.russell@gmail.com}
  {\texttt{liuch.russell@gmail.com}}}}
\affil[1]{
  Institute of Software, Chinese Academy of Sciences, Beijing, China
}
\affil[2]{
  University of Chinese Academy of Sciences, Beijing, China
}
\affil[3]{
  Department of Computer Science and Technology, Tsinghua University,
  Beijing, China
}
\date{}

\begin{document}

\maketitle

\begin{abstract}
Covariance-sensitive analyses of Gaussian annealing for sampling from a convex
body require controlling how much covariance can grow under a radial Gaussian
tilt.
For an isotropic convex body $K\subseteq\R^n$, let
$\mu_{K,t}(\dd x) \propto e^{-t\norm{x}^2}\1_K(x)\,\dd x$,
and let $Q_n$ be the supremum of $\norm{\Cov(\mu_{K,t})}_{\op}$ over all
such $K$ and all $t>0$.
We prove the sharp bound $Q_n=\Theta(n^{2/5})$, closing the gap between the
known $\Omega(n^{1/3})$ lower bound and the
$O(\sqrt{n\log(en)})$ upper bound.
The upper bound applies not only to uniform measures on convex bodies but to
every compactly supported isotropic logconcave probability measure.
It combines a dimension-free variance bound for quadratic forms with a
R\'enyi comparison at a nearby time, projected moment estimates, and
relative-entropy control along the Gaussian-tilt path.
For the matching lower bound, we construct an explicit unconditional convex
body whose axial coordinate is coupled to the transverse quadratic energy.
Moderate-deviation estimates show that an appropriate tilt creates
directional variance $\Omega(n^{2/5})$.


\end{abstract}

\newpage
\section{Introduction}\label{sec:introduction}
Uniform sampling from a convex body given only by a membership oracle is a
canonical problem in randomized algorithms and a basic primitive for volume
estimation, integration, and related high-dimensional computations
\cite{DyerFriezeKannan1991,ApplegateKannan1991}.
A central algorithmic paradigm is to reach the target distribution by
annealing through a sequence of intermediate logconcave measures, an approach
developed through the random-walk, conductance, and simulated-annealing
frameworks for convex bodies
\cite{LovaszSimonovits1993,KannanLovaszSimonovits1997,
LovaszVempala2006,LovaszVempala2007}.
The efficiency of such a scheme depends critically on how these intermediate
targets are spaced.
If consecutive targets are too far apart, a sample from
one may no longer provide a sufficiently warm start for the next; if they are
too close, the algorithm pays for too many stages.
Thus, beyond the cost of
sampling from any single target, one must understand how rapidly the
distribution can change along the entire annealing path.

Gaussian cooling gives a particularly clean realization of this paradigm
\cite{CousinsVempala2015,CousinsVempala2018}.
Starting from a sharply
concentrated Gaussian restricted to the body, one gradually removes the
quadratic penalty until reaching the uniform distribution.
The Proximal Sampler, which calls a restricted Gaussian oracle (RGO), provides
a natural way to sample from the intermediate logconcave targets
\cite{LeeShenTian2021,ChenChewiSalimWibisono2022,
KookVempala2025Cold,KookVempala2026Zeroth}.
In recent covariance-sensitive analyses, the warmness of one target relative
to the next is quantified by order-$q$ R\'enyi divergence.
Controlling this divergence along the Gaussian path reduces to controlling
radial fluctuations, which in turn depend on the covariance spectrum of the
intermediate measures
\cite{KookVempala2025Cold,KookVempala2026Zeroth}.
For an isotropic body, the uniform measure has variance one in every
direction, and each intermediate target is a radial Gaussian tilt of this
measure.
The annealing analysis requires an upper bound on the largest directional
variance.
This leads to the following stability question:
\begin{center}
  \emph{How much can the largest directional variance grow under a radial
  Gaussian tilt of an initially isotropic distribution?}
\end{center}

There is a simple geometric reason why the answer need not be
dimension-free.
Although multiplication by $e^{-t\norm{x}^2}$ favors points
closer to the origin and decreases the total quadratic energy, it need not
decrease the variance in every direction.
A convex body may contain a long
axial direction that is accessible only on transverse slices with atypically
small quadratic energy.
Under the uniform measure, such slices may carry very little mass, so the
axial marginal remains narrow.
A radial Gaussian tilt
disproportionately favors precisely these low-energy slices, thereby exposing
the long direction while shrinking many transverse directions.
Hence, the total second moment can decrease even as one covariance eigenvalue
grows.
The relevant issue is therefore not whether the tilt contracts the distribution
globally, but how much anisotropy it can create while doing so.

To make this question precise, let $K\subseteq\R^n$ be an isotropic convex
body, let $\mu_K$ denote its uniform probability measure, and for $t\geq 0$
define
\begin{equation}\label{eq:intro-tilt}
  \dd\mu_{K,t}(x)
  :=Z_{K,t}^{-1}e^{-t\norm{x}^2}\,\dd\mu_K(x),
\end{equation}
where $Z_{K,t}$ is the normalizing constant.
We consider the largest directional variance that can appear anywhere along
this path:
\begin{equation}\label{eq:intro-Qn}
  Q_n
  :=
  \sup_{\substack{K\subseteq\R^n\\
                  K\ \mathrm{isotropic}}}
  \ \sup_{t>0}
  \norm{\Cov(\mu_{K,t})}_{\op}.
\end{equation}
This quantity has a direct algorithmic interpretation.
In covariance-sensitive analyses of Gaussian cooling, particularly the
order-preserving annealing scheme of Kook and Vempala
\cite{KookVempala2025Cold,KookVempala2026Zeroth}, an upper bound on $Q_n$
controls how much the tilt parameter $t$ may change between consecutive
targets and therefore controls the covariance-dependent contribution to the
number of queries to the membership oracle.

However, the asymptotic behavior of $Q_n$ was previously unresolved.
Kook and Vempala conjectured that covariance should remain stable up to a
universal constant under quadratic tilts
\cite[Section~1.2.2]{KookVempala2025Cold}.
This would have bounded the largest covariance eigenvalue by a universal
constant throughout the path.
Subsequent work disproved such dimension-free stability and established the
following bounds
\cite[Proposition~6.3 and Lemma~4.3]{KookVempala2026Zeroth}:
\begin{equation*}
  \Omega(n^{1/3})
  \leq
  Q_n
  \leq
  O\!\left(\sqrt{n\log(en)}\right).
\end{equation*}
Related geometric constructions had already indicated that a quadratic tilt
can reveal substantial variance in a distinguished direction
\cite[Proposition~44]{Bizeul2026}.
Thus, an initially isotropic body can become polynomially anisotropic under a
radial Gaussian reweighting, but the correct scale of this instability
remained open.

Our main result determines this scale exactly:
$Q_n=\Theta(n^{2/5})$.
The upper bound in fact holds more generally for every compactly supported
isotropic logconcave probability measure, not only for uniform measures on
convex bodies.
The matching lower bound is witnessed by an explicit
\emph{unconditional} isotropic convex body, showing that covariance inflation
of order $n^{2/5}$ persists even under independent coordinate-sign
symmetries.
Thus, radial Gaussian reweighting can create polynomial
anisotropy from an isotropic starting law, and $n^{2/5}$ is the sharp
worst-case scale of this phenomenon.


\subsection{Main results}

\begin{theorem}
\label{thm:main-covariance}
There exist universal constants $c,C>0$ such that the following hold.

\begin{enumerate}
  \item[\textup{(i)}]
  Let $\mu_0$ be any compactly supported isotropic logconcave probability
  measure on $\R^n$, and for $t\geq 0$ let
  $\dd\mu_t(x)\propto e^{-t\norm{x}^2}\,\dd\mu_0(x)$.
  Then, for every $t\geq 0$,
  \begin{equation*}
    \norm{\Cov(\mu_t)}_{\op}
    \leq
    Cn^{2/5}.
  \end{equation*}

  \item[\textup{(ii)}]
  For every sufficiently large $n$, there exist an unconditional isotropic
  convex body $K_n\subseteq\R^n$ and a precision $t_n>0$ such that
  \begin{equation*}
    \norm{\Cov(\mu_{K_n,t_n})}_{\op}
    \geq
    cn^{2/5}.
  \end{equation*}
\end{enumerate}

Consequently, $Q_n=\Theta(n^{2/5})$.
\end{theorem}

The upper bound is a stability theorem for the entire Gaussian-tilt path of an
arbitrary compactly supported isotropic logconcave measure.
The lower bound already occurs for an unconditional body, that is, a body
invariant under independent sign changes of its coordinates.
Thus, the extremal covariance growth does not require a highly asymmetric
geometry: even under these coordinate symmetries, a radial Gaussian
reweighting can create a directional variance of order $n^{2/5}$.
Together, the two statements identify the exact worst-case growth of the
largest covariance eigenvalue under Gaussian tilting.

The geometric theorem feeds directly into Gaussian annealing.
We state the result for the Gaussian-annealing scheme based on the Proximal
Sampler in Kook and Vempala \cite{KookVempala2026Zeroth}.
Let $K\subseteq\R^n$ be an isotropic convex body given by a membership oracle,
assume that $K$ contains the unit ball $B_1(0)$, and let $\pi$ denote the
uniform probability measure on $K$.
For $q\geq1$ and probability measures $\nu\ll\pi$, let
$\mathrm{D}_q(\nu\|\pi)$ denote their order-$q$ R\'enyi divergence.



\subsection{Technique overview}

In this overview, $\lesssim$ and $\gtrsim$ hide universal constant factors,
and $\asymp$ denotes bounds in both directions.


\subsubsection{Upper bound}

\paragraph{Why the previous approach stops.}
When \(t\) is large, the Gaussian factor \(e^{-t\|x\|^2}\) makes the
tilted measure strongly logconcave, and Brascamp--Lieb immediately gives
\[
\|\operatorname{Cov}(\mu_t)\|_{\mathrm{op}} \lesssim t^{-1}.
\]
The difficult regime is therefore smaller \(t\), where this bound becomes too
weak. The previous \(O(\sqrt{n\log(en)})\) estimate \cite{KookVempala2026Zeroth} handles
sufficiently weak quadratic tilts using an early-stopping argument: once the
tilted distribution is \(O(1)\)-close to the initial distribution in R\'enyi
divergence, reverse H\"older transfers directional fourth moments and gives a
constant-factor covariance bound. The available R\'enyi-closeness estimate,
however, does not extend to the scale \(t\asymp n^{-2/5}\) needed here.

A different natural approach would be to propagate covariance bounds directly
from the isotropic endpoint \(t=0\), but this also loses too much. Write
\[
M_t := \mathbb{E}_t[XX^\top],
\qquad
A_t := \|M_t\|_{\mathrm{op}},
\qquad
S_t := \|M_t\|_{\mathrm{HS}}.
\]
Isotropy gives \(A_0=1\), but already \(S_0=\sqrt n\), and the local evolution
of \(A_t\) depends on \(S_t\). Our key departure is to obtain the required
second-moment bounds at a nearby earlier time \(T_-<T\), and then propagate
them only over the remaining short interval to the target \(T\).

\paragraph{Local estimates.}
The Gaussian weighting need not preserve the barycenter, so we work with the
uncentered second-moment matrix \(M_t\); since
\(\operatorname{Cov}(\mu_t)\preceq M_t\), controlling \(M_t\) suffices.
Differentiating along the path and using a non-centered consequence of
Letwin's dimension-free variance bound for quadratic forms \cite{Let26} gives
\[
\operatorname{Var}_t(\|X\|^2) \lesssim S_t^2,
\qquad
\left|\frac{\dd}{\dd t}\log S_t\right| \lesssim A_t,
\qquad
\mathcal{D}^+\log A_t \lesssim S_t,
\]
where $\mathcal{D}^+$ stands for the upper right Dini derivative.
Thus, the largest eigenvalue and the Hilbert--Schmidt norm control each other's
local evolution. These estimates are strong over short intervals, but the
initial value \(S_0=\sqrt n\) makes propagation all the way from \(t=0\) too
costly.

\paragraph{Projected second moments near the target.}
Fix a target time \(T>0\), choose \(q\ge 2\), and set
\[
T_- := T\left(1-\frac{1}{q}\right),
\qquad
H(T) := D_{\mathrm{KL}}(\mu_T\|\mu_0).
\]
The exponential-family structure of the quadratic tilt gives an exact
R\'enyi comparison between \(\mu_{T_-}\) and the initial isotropic distribution
\(\mu_0\), with the change-of-measure cost controlled by \(H(T)\).
Taking \(q\) sufficiently large relative to \(H(T)\), H\"older's inequality
and the projected Paouris moment bound \cite{Pao06,ALL+14} imply that, for
every rank-\(k\) orthogonal projection \(P\),
\[
\mathbb{E}_{T_-}\|PX\|^2 \lesssim k+q^2.
\]
Applying this estimate to the leading eigenspaces of \(M_{T_-}\) gives, in
the relevant parameter range,
\[
A_{T_-}\lesssim q^2,
\qquad
S_{T_-}\lesssim \sqrt n.
\]
The Hilbert--Schmidt bound is essential: a bound only on the largest
eigenvalue would not suffice, since the differential estimate for \(A_t\)
depends on \(S_t\). Because \(T_-\) is close to \(T\), the local estimates
then propagate these bounds only across the final short interval.

\paragraph{Relative entropy and the final scale.}
It remains to control how large \(q\) must be. The same Hilbert--Schmidt norm
that appears in the local estimates also controls the growth of the relative
entropy:
\[
H'(t)
=
t\,\operatorname{Var}_t(\|X\|^2)
\lesssim
tS_t^2.
\]
This yields a feedback between the R\'enyi comparison and the local
second-moment bounds. A first-contact argument closes the feedback: at a
hypothetical first crossing of a sufficiently large multiple of \(nt^2\),
the preceding argument gives \(S_t\lesssim\sqrt n\), and therefore
\(H'(t)\lesssim nt\), contradicting the derivative required at the crossing.
Consequently,
$H(t)\lesssim nt^2$
throughout the range needed for the argument. We may therefore choose \(q\)
on the scale \(1+nt^2\), which gives
\[
A_t
\lesssim
(1+nt^2)^2
\lesssim
1+n^2t^4.
\]

Finally, we combine this small-\(t\) estimate with Brascamp--Lieb. The two
bounds
$n^2t^4$ and
$
t^{-1}$
meet when \(t\asymp n^{-2/5}\), at which point both are of order
\(n^{2/5}\). Hence
\[
\|\operatorname{Cov}(\mu_t)\|_{\mathrm{op}}
\lesssim n^{2/5}
\]
for every \(t\ge 0\).

\subsubsection{Lower bound}

To match the upper bound, we begin with the desired behavior of the tilted
axial marginal. The condition for retaining Gaussian-scale fluctuations then
determines both the slice geometry and the critical precision.

\paragraph{Marginal criterion.}
Let \(d=n-1\), and write \((y,z)\in\R^d\times\R\) for the coordinates after
isotropization. At the precision \(t_d\) chosen below, integrating out \(y\)
shows that the \(Z\)-marginal has density proportional to
\(e^{-t_dz^2}p_d(z)\), where \(p_d(z)\) is the probability that a sample from
the tilted transverse product measure satisfies the slice constraint at height
\(z\). If \(p_d(z)\) remains bounded below by a universal constant throughout
\(\lvert z\rvert\le t_d^{-1/2}\), then the slice factor does not suppress the
Gaussian mass on its natural fluctuation scale, and
\(\Var(Z)\gtrsim t_d^{-1}\). The construction must therefore arrange two
effects that at first seem opposed. Under the uniform law on the raw body, the
axial marginal should be very narrow, so that isotropization substantially
expands raw axial distances. Under the tilt, the slices corresponding to the
resulting Gaussian window should instead acquire constant mass. The same
geometry first hides the axial direction and then allows the tilt to reveal
it.

\paragraph{Moderate-deviation geometry.}
For the moment, regard \(\Delta_d\) as a parameter and set
\(r_d:=\Delta_d^2/d\), the corresponding moderate-deviation rate. We use the
raw unconditional body
\[
K_d^{(0)}
:=
\left\{
(x,\lambda)\in[-\sqrt3,\sqrt3]^d\times\R:
\norm{x}^2+2\Delta_d\lvert\lambda\rvert
\le d-2\Delta_d
\right\}.
\]
The normalization \(\sqrt3\) makes \(d\) the mean transverse quadratic energy.
At raw height \(s=\lvert\lambda\rvert\), the relative slice volume \(G_d(s)\)
is the probability that \(\sum_{i=1}^d(V_i^2-1)\) falls below
\(-2\Delta_d(1+s)\), where the \(V_i\)'s are independent and uniform on
\([-\sqrt3,\sqrt3]\). A Cram\'er--Petrov estimate \cite{Petrov75} shows that,
for bounded \(s\), the ratio \(G_d(s)/G_d(0)\) decays on the axial scale
\(r_d^{-1}\), while a Hoeffding bound makes the remaining tail negligible in
the relevant moment integrals. The important quantity is this relative slice
profile, not the very small value of \(G_d(0)\) itself. The latter appears in
both the total mass and the axial second moment and therefore cancels from
their ratio, leaving the raw axial variance
\(b:=\Var(\lambda^{(0)})\asymp r_d^{-2}\). Although the retained slices are
rare, their points are not concentrated near the origin; a separate
concentration argument shows that the transverse variance
\(a:=\Var(X_1^{(0)})\) remains of constant order. Hence diagonal
isotropization gives \(z=b^{-1/2}\lambda\) with
\(b^{-1/2}\asymp r_d\), while changing the transverse coordinates only by
constant factors. This is the key geometric gain. Bounded raw heights become
an axial interval of width \(\asymp r_d\) in isotropic coordinates.

\paragraph{Critical scale.}
The tilt now exploits this dilation. Since isotropization changes the
transverse scale only by a constant factor, a tilt of precision \(t_d\) lowers
the mean transverse quadratic energy by order \(dt_d\). It can therefore make
the deficient slices typical once \(dt_d\gtrsim\Delta_d\). On the other hand,
since \(z\asymp r_d\lambda\), the Gaussian window
\(\lvert z\rvert\le t_d^{-1/2}\) corresponds to raw heights of order
\((r_d\sqrt{t_d})^{-1}\); keeping these heights bounded requires
\(t_d\gtrsim r_d^{-2}\). The first condition opens the slices, while the
second ensures that the opened range covers the Gaussian window. Within this
scheme, minimizing the admissible precision amounts to balancing
\(dt_d\asymp\Delta_d\) and \(t_d\asymp r_d^{-2}\). Using
\(r_d=\Delta_d^2/d\) gives \(\Delta_d^5\asymp d^3\). Accordingly, the proof
sets \(\Delta_d:=d^{3/5}\), \(r_d=d^{1/5}\), and
\(t_d:=A r_d^{-2}\asymp d^{-2/5}\) for a sufficiently large universal
constant \(A\). This choice lies genuinely in the moderate-deviation regime
\(\sqrt d\ll\Delta_d\ll d\). At this precision, Hoeffding concentration gives
\(p_d(z)\ge1/2\) throughout the Gaussian window, and the marginal criterion
yields \(\Var(Z)\gtrsim t_d^{-1}\asymp d^{2/5}\asymp n^{2/5}\). Since the
covariance operator norm dominates this axial variance, the desired lower
bound follows.

The same calculation also explains the previous \(d^{1/3}\) lower bound. In
Bizeul's construction \cite[Proposition~44]{Bizeul2026}, later put into
isotropic position and analyzed by Kook and Vempala
\cite[Proposition~6.3]{KookVempala2026Zeroth}, the raw axial and transverse
variances are both of constant order, so isotropization introduces no
dimension-dependent dilation. Its radial cutoff also changes at constant
speed with the axial coordinate. A shift of order \(dt\) in squared radius
therefore changes the typical radius, and hence the useful axial window, by
only order \(t\sqrt d\). The resulting variance scale is
\(\min\{1+t^2d,t^{-1}\}\), whose maximum is of order \(d^{1/3}\) at
\(t\asymp d^{-1/3}\). Thus, the earlier exponent is intrinsic to that geometry
rather than an artifact of a loose estimate. Our construction changes
precisely this ingredient. Placing the central slice at a moderate-deviation
depth makes the raw axial standard deviation \(r_d^{-1}\), so isotropization
creates the growing dilation \(r_d\) that changes the balance from \(1/3\) to
\(2/5\).

\subsection{Related work}\label{sec:related-work}


\paragraph{Sampling and annealing.}
The membership-oracle approach to convex-body computation was initiated by
Dyer, Frieze, and Kannan and extended to sampling and integration of
near-logconcave functions by Applegate and Kannan
\cite{DyerFriezeKannan1991,ApplegateKannan1991}.
Conductance and localization analyses of the ball walk, followed by hit-and-run
and simulated annealing, established the framework of annealing through a
sequence of logconcave targets
\cite{LovaszSimonovits1993,KLS1995,KannanLovaszSimonovits1997,
KannanLovaszMontenegro2006,LovaszVempala2006,LovaszVempala2007}.
Affine-invariant Dikin walks and Gaussian cooling refined the geometry and led
to faster volume algorithms
\cite{KannanNarayanan2012,LaddhaLeeVempala2020,CousinsVempala2015,
CousinsVempala2018}.
Algorithmic diffusion, restricted Gaussian oracles, and the Proximal Sampler
now provide complementary routes to sampling and integration
\cite{KookVempala2024Dikin,KookVempalaZhang2024,KookVempala2025Diffusion,
LeeShenTian2021,ChenChewiSalimWibisono2022}.
Recent work addresses cold starts and membership-only access using order-$q$
R\'enyi warm starts, constrained variants of the Proximal Sampler, and
zeroth-order algorithms
\cite{KookVempala2025Cold,KookZhang2025,NarayananRajaramanSrivastava2025,
KookVempala2026Zeroth}.
These papers analyze mixing and warmness for a chosen schedule; $Q_n$ isolates
the covariance growth that governs the warmness of adjacent targets along the
radial Gaussian path.

\paragraph{Geometry of isotropic logconcave measures.}
For fixed isotropic logconcave laws, Brascamp--Lieb controls covariance in the
strongly logconcave regime, while Paouris and Adamczak et al.\ provide sharp
radial tails and moment bounds \cite{BrascampLieb1976,Pao06,ALL+14}.
Thin-shell and large-deviation estimates of Gu\'edon--Milman and Klartag's
central-limit theorem describe radial and projection fluctuations
\cite{GuedonMilman2011,Klartag2007}.
Stochastic localization connects thin-shell information to spectral-gap and
log-Sobolev inequalities \cite{Eldan2013,LeeVempala2017,LeeVempala2018}.
Letwin's quadratic-form inequality and recent slicing/KLS bounds sharpen these
estimates for individual measures \cite{Let26,Klartag2023}.

\paragraph{Perturbations and covariance stability.}
Gaussian comparison and perturbation results of Caffarelli, Harg\'e,
Cattiaux--Guillin, and Klartag--Putterman concern transport, functional
inequalities, or Gaussian convolution
\cite{Caffarelli2000,Harge2004,CattiauxGuillin2022,KlartagPutterman2023}.
Bizeul's one-constraint example is the closest precursor to our construction:
it exhibits polynomial directional variance under a quadratic tilt
\cite[Proposition~44]{Bizeul2026}.
These results concern a single measure, a Gaussian convolution, or a different
functional-inequality question.
In contrast, $Q_n$ takes a supremum over the Gaussian-tilt path: it measures the
largest directional variance created by multiplying one isotropic law by a
radial Gaussian factor.
It is therefore related to, but distinct from, KLS and thin-shell parameters,
and it is the geometric quantity that enters R\'enyi-warmness bounds for
Gaussian annealing.


\section{Notation and preliminaries}\label{sec:preliminaries}

\subsection{Notation and conventions}\label{sec:notation}

We first fix some conventions.
We use $\norm{x}$ to denote the Euclidean norm of a vector $x$, and $B_r(x)$ to denote the closed Euclidean ball with center
$x$ and radius $r$.  We write $\log_+(u)=\max\{0,\log u\}$.
For a real symmetric matrix $A$, let $\norm{A}_{\op}$ denote its operator norm and $\Tr(A)$ its
trace, and write $\norm{A}_{\mathrm{HS}}:=\sqrt{\Tr(A^2)}$.  We write $A\preceq B$ if $B-A$ is
positive semidefinite.  All matrix square roots below are the positive-semidefinite ones.

\paragraph{Probability theory.}
A probability measure on $\R^d$ is \emph{logconcave} if its density has the form $e^{-V}$ for an
extended-valued convex function $V$.  If the measure lies in a proper affine subspace, this
definition is understood relative to that subspace.

Let $\nu$ be a probability measure on $\R^d$, and let
$X=(X_1,\ldots,X_d)^{\mathsf T}$ be an $\R^d$-valued random column vector with law $\nu$.  We
write this relation as $X\sim\nu$.  For an integrable function $f:\R^d\to\R$, set
\[
  \E_\nu[f]:=\int_{\R^d}f(x)\,\dd\nu(x)=\E[f(X)].
\]
The law $\nu$, or equivalently the random vector $X$, is \emph{isotropic} if
$\E_\nu[X]=0$ and $\Cov_\nu(X)=I_d$.

For square-integrable functions $f,g:\R^d\to\R$, set
\[
  \Cov_\nu(f,g):=\E_\nu[fg]-\E_\nu[f]\,\E_\nu[g],
  \qquad
  \Var_\nu(f):=\Cov_\nu(f,f).
\]
Covariance involving a vector- or matrix-valued function is interpreted entrywise.

For the random column vector $X\sim\nu$, define
\begin{equation}\label{eq:moment-notation}
  m_\nu:=\E_\nu[X],
  \qquad
  \Sigma_\nu:=\Cov_\nu(X),
  \qquad
  M_\nu:=\E_\nu[XX^{\mathsf T}]
  =\Sigma_\nu+m_\nu m_\nu^{\mathsf T}.
\end{equation}
Here $X^{\mathsf T}$ is a row vector, so
$XX^{\mathsf T}=(X_iX_j)_{i,j=1}^d$ is the $d\times d$ outer-product matrix. $\Sigma_\nu$
is the covariance matrix, while $M_\nu$ is the uncentered second-moment matrix.
We also write $\Cov(\nu):=\Sigma_\nu$, and omit subscripts when the underlying law is clear.

For probability measures $\nu\ll\mu$, their Kullback--Leibler divergence is
\[
  \KL(\nu\|\mu)
  :=\int\log\left(\frac{\dd\nu}{\dd\mu}\right)\dd\nu.
\]
For $q>1$, their order-$q$ R\'enyi divergence is
\begin{equation}\label{eq:renyi-definition}
  \mathrm{D}_q(\nu\|\mu)
  :=\frac1{q-1}\log\int
  \left(\frac{\dd\nu}{\dd\mu}\right)^q \dd\mu.
\end{equation}
Both divergences are $+\infty$ when $\nu\not\ll\mu$, and the continuous extension at $q=1$ is
$\mathrm{D}_1(\nu\|\mu)=\KL(\nu\|\mu)$.

\subsection{Gaussian tilts and basic calculus}\label{sec:tilt-calculus}

Let $\mu_0$ be a compactly supported probability measure on $\R^n$.  Its Gaussian tilt at
precision $t\in\R$ is the probability measure $\mu_t$ defined by
\begin{equation}\label{eq:gaussian-tilt}
  Z(t):=\E_0[e^{-t\norm{X}^2}],
  \qquad
  \phi(t):=\log Z(t),
  \qquad
  \frac{\dd\mu_t}{\dd\mu_0}(x)=e^{-t\norm{x}^2-\phi(t)}.
\end{equation}
Equivalently, for every measurable set $A\subseteq\R^n$,
\[
  \mu_t(A)=\frac{1}{Z(t)}\int_A e^{-t\norm{x}^2}\,\dd\mu_0(x).
\]
We write $\E_t$, $\Cov_t$, and $\Var_t$ for expectation, covariance, and variance under $\mu_t$.
Likewise, $m_t$, $\Sigma_t$, and $M_t$ denote the quantities in \eqref{eq:moment-notation} for
$\nu=\mu_t$.

The following identities collect the elementary calculus used along the tilt path.

\begin{lemma}[Gaussian-tilt identities]\label{lem:tilt-calculus}
For every $\mu_0$-integrable function $f$ and every $t\ge0$,
\begin{align}
  \frac{\dd}{\dd t}\E_t[f]&=-\Cov_t(f,\norm{X}^2),\label{eq:flow-basic}\\
  \frac{\dd}{\dd t}\phi(t)&=-\E_t[\norm{X}^2],
  &\frac{\dd^2}{\dd t^2}\phi(t)&=\Var_t(\norm{X}^2).\label{eq:phi-derivatives}
\end{align}
In particular,
\begin{equation}\label{eq:moment-flow-basic}
  \frac{\dd}{\dd t}m_t=-\Cov_t(X,\norm{X}^2),
  \qquad
  \frac{\dd}{\dd t}M_t=-\Cov_t(XX^{\mathsf T},\norm{X}^2).
\end{equation}

For every $r>1$ and $s,t\in\R$,
\begin{equation}\label{eq:renyi-tilt-general}
  \mathrm{D}_r(\mu_s\|\mu_t)
  =\frac{\phi(t+r(s-t))-r\phi(s)+(r-1)\phi(t)}{r-1}.
\end{equation}
\end{lemma}

\begin{proof}
Choose $R<\infty$ such that $\norm{x}\le R$ for $\mu_0$-almost every $x$.  On every bounded
interval of values of $t$, both $e^{-t\norm{x}^2}$ and
$\norm{x}^2e^{-t\norm{x}^2}$ are uniformly bounded for $\norm{x}\le R$.  Thus, for every
$\mu_0$-integrable $f$, dominated convergence permits differentiation under the integral.
In particular,
\[
  \frac{\dd}{\dd t}Z(t)
  =-\E_0\bigl[\norm{X}^2e^{-t\norm{X}^2}\bigr].
\]
\begin{samepage}
Since $Z(t)>0$, the chain rule and the definition of $\mu_t$ in
\eqref{eq:gaussian-tilt} give
\begin{align*}
  \frac{\dd}{\dd t}\phi(t)
  &=\frac{1}{Z(t)}\frac{\dd}{\dd t}Z(t)\\
  &=-\frac{1}{Z(t)}
    \E_0\bigl[\norm{X}^2e^{-t\norm{X}^2}\bigr]\\
  &=-\int_{\R^n}\norm{x}^2
    \frac{e^{-t\norm{x}^2}}{Z(t)}\,\dd\mu_0(x)\\
  &=-\int_{\R^n}\norm{x}^2\,\dd\mu_t(x)\\
  &=-\E_t[\norm{X}^2].
\end{align*}
\end{samepage}
Moreover,
\[
  \E_t[f]
  =e^{-\phi(t)}\E_0\bigl[f e^{-t\norm{X}^2}\bigr].
\]
Differentiating both factors and substituting the preceding identity yields
\begin{align*}
  \frac{\dd}{\dd t}\E_t[f]
  &=-\frac{\dd}{\dd t}\phi(t)\,\E_t[f]
    -\E_t[f\norm{X}^2]\\
  &=\E_t[f]\E_t[\norm{X}^2]-\E_t[f\norm{X}^2]\\
  &=-\Cov_t(f,\norm{X}^2).
\end{align*}
Applying this identity with $f=\norm{X}^2$ gives
\[
  \frac{\dd^2}{\dd t^2}\phi(t)
  =-\frac{\dd}{\dd t}\E_t[\norm{X}^2]
  =\Var_t(\norm{X}^2).
\]
Applying it coordinatewise with $f=X$ and $f=XX^{\mathsf T}$ gives the two identities in
\eqref{eq:moment-flow-basic}.

It remains to compute the R\'enyi divergence.  The densities in \eqref{eq:gaussian-tilt} are
strictly positive $\mu_0$-almost everywhere, so $\mu_s$ and $\mu_t$ are mutually absolutely
continuous, with
\[
  \frac{\dd\mu_s}{\dd\mu_t}(x)
  =\exp\bigl(-(s-t)\norm{x}^2-\phi(s)+\phi(t)\bigr).
\]
Therefore,
\begin{align*}
  \int\left(\frac{\dd\mu_s}{\dd\mu_t}\right)^r\dd\mu_t
  &=e^{-r\phi(s)+r\phi(t)}
    \int e^{-r(s-t)\norm{x}^2}\dd\mu_t(x)\\
  &=e^{-r\phi(s)+(r-1)\phi(t)}
    \int e^{-(t+r(s-t))\norm{x}^2}\dd\mu_0(x)\\
  &=\exp\bigl(\phi(t+r(s-t))-r\phi(s)+(r-1)\phi(t)\bigr).
\end{align*}
Taking the logarithm and dividing by $r-1$ proves \eqref{eq:renyi-tilt-general}.
\end{proof}

The entropy of the endpoint relative to the initial measure will be denoted by
\begin{equation}\label{eq:entropy-definition}
  H(t):=\KL(\mu_t\|\mu_0)
  =t\phi'(t)-\phi(t).
\end{equation}
By \cref{lem:tilt-calculus},
\begin{equation}\label{eq:entropy-derivative}
  H'(t)=t\Var_t(\norm{X}^2).
\end{equation}
Thus, $\phi$ is convex, while $H$ is nonnegative and nondecreasing on $[0,\infty)$.

\subsection{Useful Inequalities}\label{sec:analytic-inputs}

We use three standard estimates.  The first is Letwin's dimension-free variance inequality for
quadratic forms.

\begin{theorem}[{\cite[Theorem 1.2]{Let26}}]\label{thm:letwin-quadratic}
Let $Y$ be an isotropic logconcave random vector in $\R^d$.  For every symmetric matrix
$A\in\R^{d\times d}$, let $f_A(y)=y^{\mathsf T}Ay$.  Here
$\nabla(Y^{\mathsf T}AY)$ is shorthand for the Euclidean gradient
$\nabla f_A(Y)=2AY$, evaluated at $Y$.  Then
\begin{equation}\label{eq:letwin-quadratic}
  \Var(Y^{\mathsf T}AY)
  \le2\E[\norm{\nabla(Y^{\mathsf T}AY)}^2]
  =8\Tr(A^2).
\end{equation}
\end{theorem}

The final equality in \eqref{eq:letwin-quadratic} is immediate from isotropy:
\[
  \E[\norm{\nabla f_A(Y)}^2]
  =4\E[Y^{\mathsf T}A^2Y]
  =4\Tr(A^2).
\]
The inequality itself is the main result of \cite[Theorem 1.2]{Let26}.


The second is Paouris's norm-tail estimate and the projected positive-moment consequence that
we use.

\begin{theorem}[Paouris tail bound {\cite[Theorem 1.1 and Section 8]{Pao06}}]
\label{thm:paouris-tail}
Let $Z$ be an isotropic logconcave random vector in $\R^k$.  There is a universal constant
$C_0>0$ such that, for every $t\ge1$,
\begin{equation}\label{eq:paouris-tail}
  \Pp\bigl(\norm{Z}\ge C_0t\sqrt{k}\bigr)\le e^{-t\sqrt{k}}.
\end{equation}
\end{theorem}

The cited theorem is stated for the uniform measure on an isotropic convex body;
\cite[Section 8]{Pao06} records its extension to arbitrary isotropic logconcave measures.

\begin{theorem}[Projected Paouris bound]\label{thm:paouris}
Let $X$ be isotropic and logconcave in $\R^n$, and let $P$ be an orthogonal projection of rank
$k$.  For every $p\ge2$,
\begin{equation}\label{eq:paouris-projection}
  \bigl(\E[\norm{PX}^p]\bigr)^{1/p}\le C(\sqrt{k}+p).
\end{equation}
\end{theorem}

\begin{proof}
Let $F=\operatorname{range}(P)$ and regard $Z=PX$ as a random vector in the Euclidean space
$F\cong\R^k$.  First, $Z$ is logconcave.  Indeed, linear images of logconcave measures are
logconcave.  More explicitly, if $f$ is the density of $X$, then the density of $Z$ on $F$ is
\[
  g(z)=\int_{F^\perp}f(z+w)\,\dd w,
  \qquad z\in F,
\]
which is logconcave by Pr\'ekopa's theorem.

Next, $Z$ is isotropic on $F$.  Since $X$ is isotropic,
\[
  \E[Z]=P\E[X]=0.
\]
For every $u,v\in F$, orthogonality of $P$ gives $Pu=u$ and $Pv=v$, and hence
\begin{align*}
  \E[\langle Z,u\rangle\langle Z,v\rangle]
  &=\E[\langle PX,u\rangle\langle PX,v\rangle]\\
  &=\E[\langle X,Pu\rangle\langle X,Pv\rangle]\\
  &=\E[\langle X,u\rangle\langle X,v\rangle]
    =\langle u,v\rangle.
\end{align*}
Thus the covariance of $Z$ as an $F$-valued random vector is $I_F$.  Notice that its covariance
as an $\R^n$-valued random vector is $P$; restricting to $F$ is what makes it isotropic.

We may therefore apply \cref{thm:paouris-tail} in dimension $k$.  Writing $R=\norm{Z}$ and
setting $r=C_0t\sqrt{k}$ in \eqref{eq:paouris-tail}, we obtain
\[
  \Pp(R\ge r)\le e^{-r/C_0},
  \qquad r\ge C_0\sqrt{k}.
\]
Set $r_0:=C_0\sqrt{k}$.  Since $R\ge0$, the pointwise identity
\begin{equation*}
  R^p=\int_0^\infty p r^{p-1}\1_{\{R\ge r\}}\,\dd r
\end{equation*}
and Tonelli's theorem give, for $p\ge2$,
\begin{align*}
  \E[R^p]
  &=p\int_0^\infty r^{p-1}\Pp(R\ge r)\,\dd r\\
  &=p\int_0^{r_0}r^{p-1}\Pp(R\ge r)\,\dd r
    +p\int_{r_0}^\infty r^{p-1}\Pp(R\ge r)\,\dd r\\
  &\le p\int_0^{r_0}r^{p-1}\,\dd r
    +p\int_{r_0}^\infty r^{p-1}e^{-r/C_0}\,\dd r\\
  &\le r_0^p+p\int_0^\infty r^{p-1}e^{-r/C_0}\,\dd r.
\end{align*}
In the first inequality, we use $\Pp(R\ge r)\le1$ for $0\le r\le r_0$ and the
preceding tail bound for $r\ge r_0$.
The last inequality enlarges the integration interval of a nonnegative
integrand.
Substituting $u=r/C_0$ in the remaining integral gives
\begin{align*}
  p\int_0^\infty r^{p-1}e^{-r/C_0}\,\dd r
  &=pC_0^p\int_0^\infty u^{p-1}e^{-u}\,\dd u\\
  &=pC_0^p\Gamma(p)
   =C_0^p\Gamma(p+1).
\end{align*}
The second line uses the definition of the gamma function and its recurrence
$\Gamma(p+1)=p\Gamma(p)$.
Consequently,
\begin{equation*}
  \E[R^p]\le(C_0\sqrt{k})^p+C_0^p\Gamma(p+1).
\end{equation*}
First, the upper form of Stirling's estimate gives a universal constant
$C_1\ge1$ such that, for every $p\ge2$,
\begin{equation*}
  \Gamma(p+1)\le C_1\sqrt{p}\left(\frac{p}{e}\right)^p.
\end{equation*}
Therefore,
\begin{align*}
  \sqrt[p]{\Gamma(p+1)}
  &\le C_1^{1/p}p^{1/(2p)}\frac{p}{e}\\
  &\le C_2p,
\end{align*}
where $C_2$ is universal, since both $C_1^{1/p}$ and $p^{1/(2p)}$ are
uniformly bounded for $p\ge2$.  For $p\ge1$ and $a,b\ge0$, we also have
\begin{equation*}
  (a^p+b^p)^{1/p}\le a+b.
\end{equation*}
\begin{samepage}
Applying this inequality to the preceding moment bound and recalling that
$R=\norm{Z}=\norm{PX}$, we obtain
\begin{align*}
  \bigl(\E[\norm{PX}^p]\bigr)^{1/p}
  &=\bigl(\E[R^p]\bigr)^{1/p}\\
  &\le\bigl((C_0\sqrt{k})^p+C_0^p\Gamma(p+1)\bigr)^{1/p}\\
  &\le C_0\sqrt{k}+C_0\sqrt[p]{\Gamma(p+1)}\\
  &\le C_0\sqrt{k}+C_0C_2p\\
  &\le C(\sqrt{k}+p),
\end{align*}
as claimed.
\end{samepage}
\end{proof}

The final input controls the strongly Gaussian part of the tilt path.  We recall the precise
nonsmooth consequence of the Brascamp--Lieb inequality that we need.

\begin{theorem}[Brascamp--Lieb covariance bound]\label{thm:brascamp-lieb}
Let $W:\R^d\to\R\cup\{+\infty\}$ be proper and lower semicontinuous, and suppose that
\[
  0<\mathcal Z_W:=\int_{\R^d}e^{-W(x)}\,\dd x<\infty.
\]
For some $\kappa>0$, assume that $W$ is $\kappa$-strongly convex in the extended-valued sense;
that is, $x\mapsto W(x)-\frac{\kappa}{2}\norm{x}^2$ is convex.  If $\nu$ is the probability
measure
\[
  \dd\nu(x)=\mathcal Z_W^{-1}e^{-W(x)}\,\dd x,
\]
then
\begin{equation}\label{eq:BL-general}
  \Cov(\nu)\preceq\frac1\kappa I_d.
\end{equation}
Equivalently, for every $u\in\R^d$,
\[
  \Var_\nu(\langle u,X\rangle)\le\frac1\kappa\norm{u}^2.
\]
\end{theorem}

\begin{proof}
First suppose that $W$ is finite and smooth.  Strong convexity gives
$\nabla^2W(x)\succeq\kappa I_d$.  The Brascamp--Lieb variance inequality
\cite[Theorem 4.1]{BrascampLieb1976}, applied to $h_u(x)=\langle u,x\rangle$, yields
\[
  \Var_\nu(h_u)
  \le\int_{\R^d}
    \langle(\nabla^2W(x))^{-1}u,u\rangle\,\dd\nu(x)
  \le\frac1\kappa\norm{u}^2.
\]
This is equivalent to \eqref{eq:BL-general}.

For a proper lower-semicontinuous extended-valued $W$, apply Moreau-envelope regularization and
mollification to the convex function $W-\frac{\kappa}{2}\norm{\cdot}^2$, and then add back
$\frac{\kappa}{2}\norm{\cdot}^2$.  This gives smooth $\kappa$-strongly convex potentials whose
normalized measures converge to $\nu$ with their moments through degree two.  Applying the smooth
case and passing to the limit proves \eqref{eq:BL-general}.
\end{proof}

To apply the theorem, write the logconcave density of $\mu_0$ as $e^{-V}$, with $V$ convex and
possibly extended-valued.  The potential of $\mu_t$ is
\[
  W_t(x)=V(x)+t\norm{x}^2,
\]
and $W_t-\frac{2t}{2}\norm{\cdot}^2=V$ is convex.  Thus $W_t$ is $2t$-strongly convex, and
\cref{thm:brascamp-lieb} with $\kappa=2t$ gives
\begin{equation}\label{eq:BL}
  \Sigma_t=\Cov(\mu_t)\preceq\frac1{2t}I,
  \qquad t>0.
\end{equation}

\subsection{A non-centered quadratic-form estimate}\label{sec:noncentered-quadratic-form}

We record a consequence of Letwin's theorem for a logconcave vector that is not necessarily
centered or isotropic.  The estimate is stated in terms of the uncentered matrix $M$ from
\eqref{eq:moment-notation}.

\begin{lemma}[Variance of a non-centered quadratic form]\label{lem:quadratic-variance}
Let $X\in\R^d$ have a nondegenerate logconcave law, and put $M=\E[XX^{\mathsf T}]$.  For every
symmetric matrix $B$,
\begin{equation}\label{eq:quadratic-variance}
  \Var(X^{\mathsf T}BX)
  \le10\Tr\!\left((M^{1/2}BM^{1/2})^2\right)
  =10\Tr((BM)^2).
\end{equation}
\end{lemma}

\begin{proof}
Let $m=\E[X]$ and $\Sigma=\Cov X$.  Write $X=m+\Sigma^{1/2}Y$, where $Y$ is isotropic and
logconcave, and set
\[
  A:=\Sigma^{1/2}B\Sigma^{1/2},
  \qquad
  b:=\Sigma^{1/2}Bm.
\]
Then
\[
  X^{\mathsf T}BX
  =Y^{\mathsf T}AY+2b^{\mathsf T}Y+m^{\mathsf T}Bm.
\]
By \cref{thm:letwin-quadratic},
$\Var(Y^{\mathsf T}AY)\le8\Tr(A^2)$.  Isotropy also gives the exact identity
$\Var(2b^{\mathsf T}Y)=4\norm{b}^2$.

The triangle inequality in $L^2$ now gives
\[
  \sqrt{\Var(X^{\mathsf T}BX)}
  \le\sqrt{8\Tr(A^2)}+2\norm{b}.
\]
Since
\[
  \left(\sqrt{8\Tr(A^2)}+2\norm{b}\right)^2
  \le10\left(\Tr(A^2)+2\norm{b}^2\right),
\]
it suffices to show that
$\Tr(A^2)+2\norm{b}^2\le\Tr((BM)^2)$.

Using $M=\Sigma+mm^{\mathsf T}$, we obtain
\[
  \Tr((BM)^2)
  =\Tr(B\Sigma B\Sigma)
   +2m^{\mathsf T}B\Sigma Bm
   +(m^{\mathsf T}Bm)^2
  =\Tr(A^2)+2\norm{b}^2+(m^{\mathsf T}Bm)^2.
\]
This proves \eqref{eq:quadratic-variance}.
\end{proof}

\section{The upper bound}\label{sec:upper}

\begin{theorem}[Upper bound for logconcave measures]\label{thm:upper-general}
Let $\mu_0$ be a compactly supported isotropic logconcave probability measure on $\R^n$.  Then its
Gaussian tilts from \eqref{eq:gaussian-tilt} satisfy
\begin{equation}\label{eq:upper-logconcave}
  \norm{\Sigma_t}_{\op}\le Cn^{2/5},
  \qquad t\ge0,
\end{equation}
where $C>0$ is universal.
\end{theorem}

Throughout this section, we fix $\mu_0$ as in \cref{thm:upper-general} and use the notation and
preliminary results from \cref{sec:preliminaries}.

\subsection{Differential estimates along the tilt}\label{sec:upper-differential}

For the upper-bound argument, set
\begin{equation}\label{eq:upper-state}
  m_t:=\E_t[X],
  \qquad
  \Sigma_t:=\Cov_t(X),
  \qquad
  M_t:=\E_t[XX^{\mathsf T}],
  \qquad
  S_t:=\norm{M_t}_{\mathrm{HS}}=\sqrt{\Tr(M_t^2)}.
\end{equation}
We use $M_t$ rather than $\Sigma_t$ because differentiation under a radial tilt closes directly on
uncentered quadratic forms.  Moreover, $\Sigma_t\preceq M_t$, so $\norm{M_t}_{\op}$ controls the
quantity in \eqref{eq:upper-logconcave}.

The law $\mu_t$ is nondegenerate, hence $M_t$ is positive definite and $S_t>0$.  By
\eqref{eq:moment-flow-basic},
\begin{equation}\label{eq:M-derivative}
  M_t'=-\Cov_t(XX^{\mathsf T},\norm{X}^2).
\end{equation}
For a real-valued function $f$, write
\begin{equation}\label{eq:dini-definition}
  \mathcal D^+f(t):=\limsup_{h\downarrow0}\frac{f(t+h)-f(t)}h
\end{equation}
for its upper right Dini derivative.

\begin{lemma}[Differential estimates]\label{lem:differential-estimates}
For every $t\ge0$,
\begin{align}
  \Var_t(\norm{X}^2)&\le10S_t^2,\label{eq:var-energy}\\
  \left|\frac{\dd}{\dd t}\log S_t\right|&\le10\norm{M_t}_{\op},\label{eq:logS-flow}\\
  \mathcal D^+\bigl(\log\norm{M_\cdot}_{\op}\bigr)(t)&\le10S_t.\label{eq:log-Mop-flow}
\end{align}
\end{lemma}

\begin{proof}
Applying \cref{lem:quadratic-variance} with $B=I_n$ gives \eqref{eq:var-energy}.

Next, the definition of $S_t$ in \eqref{eq:upper-state} gives
$S_t^2=\Tr(M_t^2)$.  Since $M_t$ is differentiable and $S_t>0$, the chain
rule and cyclicity of the trace give
\begin{align*}
  2S_tS_t'
  &=\frac{\dd}{\dd t}S_t^2\\
  &=\frac{\dd}{\dd t}\Tr(M_t^2)\\
  &=\Tr(M_t'M_t+M_tM_t')\\
  &=2\Tr(M_tM_t').
\end{align*}
For any deterministic matrix $B$, linearity of trace, expectation, and
covariance gives
\begin{align*}
  \Tr\!\left(B\Cov_t(XX^{\mathsf T},\norm{X}^2)\right)
  &=\Cov_t\bigl(\Tr(BXX^{\mathsf T}),\norm{X}^2\bigr)\\
  &=\Cov_t(X^{\mathsf T}BX,\norm{X}^2).
\end{align*}
At each fixed $t$, the matrix $M_t$ is deterministic.  Substituting $B=M_t$
and \eqref{eq:M-derivative} into the preceding calculation yields
\begin{align*}
  2S_tS_t'
  &=2\Tr(M_tM_t')\\
  &=-2\Tr\!\left(M_t\Cov_t(XX^{\mathsf T},\norm{X}^2)\right)\\
  &=-2\Cov_t(X^{\mathsf T}M_tX,\norm{X}^2).
\end{align*}
Cauchy--Schwarz and \cref{lem:quadratic-variance} imply
\begin{align*}
  2S_t|S_t'|
  &\le2\sqrt{\Var_t(X^{\mathsf T}M_tX)\Var_t(\norm{X}^2)}\\
  &\le20\sqrt{\Tr(M_t^4)S_t^2}
  \le20\norm{M_t}_{\op}S_t^2.
\end{align*}
Dividing by $2S_t^2$ proves \eqref{eq:logS-flow}.

Let $E_t$ be the top eigenspace of $M_t$.  The variational formula for the largest eigenvalue of a
differentiable symmetric matrix gives
\[
  \mathcal D^+\bigl(\norm{M_\cdot}_{\op}\bigr)(t)\le
  \max_{\substack{v\in E_t\\ \norm{v}=1}}v^{\mathsf T}M_t'v.
\]
For each unit vector $v\in E_t$, \eqref{eq:M-derivative} and
\cref{lem:quadratic-variance} give
\begin{align*}
  |v^{\mathsf T}M_t'v|
  &=\left|\Cov_t((v^{\mathsf T}X)^2,\norm{X}^2)\right|\\
  &\le\sqrt{10\norm{M_t}_{\op}^2\cdot10S_t^2}
  =10\norm{M_t}_{\op}S_t.
\end{align*}
Taking the maximum and dividing by $\norm{M_t}_{\op}$ proves \eqref{eq:log-Mop-flow}.
\end{proof}

\begin{corollary}[Integrated differential estimates]\label{cor:integrated-differential}
For every $0\le a<b$,
\begin{align}
  S_b&\le S_a\exp\left(10\int_a^b\norm{M_s}_{\op}\,\dd s\right),\label{eq:S-integrated}\\
  \norm{M_b}_{\op}&\le \norm{M_a}_{\op}
  \exp\left(10\int_a^b S_s\,\dd s\right).\label{eq:Mop-integrated}
\end{align}
\end{corollary}

\begin{proof}
Integrating \eqref{eq:logS-flow} proves \eqref{eq:S-integrated}.  The map
$t\mapsto\norm{M_t}_{\op}$ is locally Lipschitz because $M_t$ is continuously differentiable and
the top eigenvalue is $1$-Lipschitz in operator norm.
Thus the Dini-derivative bound \eqref{eq:log-Mop-flow} can be integrated, which gives
\eqref{eq:Mop-integrated}.
\end{proof}

\subsection{Estimates at an earlier time}\label{sec:upper-earlier-time}

We next compare a target time $T$ with a nearby earlier time.  The change of measure and the
spectral estimate are kept separate.

\begin{lemma}[R\'enyi comparison]\label{lem:renyi-comparison}
Let $T>0$ and $q>1$, and define $q':=q/(q-1)$.  Then
\begin{equation}\label{eq:renyi-entropy}
  \mathrm{D}_{q'}(\mu_{T/q'}\|\mu_0)
  =\frac{\phi(T)-q'\phi(T/q')}{q'-1}
  \le H(T).
\end{equation}
\end{lemma}

\begin{proof}
Set $s:=T/q'$.  By the definition of the Gaussian tilt and the fact that $\phi(0)=0$,
the density of $\mu_s$ with respect to $\mu_0$ is
\[
  \frac{\dd\mu_s}{\dd\mu_0}(x)
  =\exp\bigl(-s\norm{x}^2-\phi(s)\bigr).
\]
Consequently,
\begin{align*}
  \int\left(\frac{\dd\mu_s}{\dd\mu_0}\right)^{q'}\dd\mu_0
  &=e^{-q'\phi(s)}\E_0\bigl[e^{-q's\norm{X}^2}\bigr]\\
  &=\exp\bigl(\phi(q's)-q'\phi(s)\bigr)\\
  &=\exp\bigl(\phi(T)-q'\phi(T/q')\bigr),
\end{align*}
where the last equality uses $q's=T$.  Taking the logarithm and dividing by $q'-1$ as in
\eqref{eq:renyi-definition} proves the identity in \eqref{eq:renyi-entropy}.

It remains to prove the inequality.  Since $s<T$ and $\phi$ is convex, its supporting line at
$T$ lies below its graph:
\[
  \phi(s)\ge \phi(T)+\phi'(T)(s-T).
\]
Multiplying by $-q'$ and using $s=T/q'$ gives
\begin{align*}
  \phi(T)-q'\phi(s)
  &\le \phi(T)-q'\phi(T)-q'\phi'(T)(s-T)\\
  &=(q'-1)\bigl(T\phi'(T)-\phi(T)\bigr)\\
  &=(q'-1)H(T).
\end{align*}
Dividing by $q'-1>0$ completes the proof.
\end{proof}

\begin{lemma}[Projected second-moment estimate]\label{lem:projected-second-moment}
Let $T>0$, $q\ge\max\{2,H(T)\}$, and $T_-=T(1-1/q)$.  For every orthogonal projection $P$ of
rank $k$,
\begin{equation}\label{eq:projected-second-moment}
  \Tr(PM_{T_-})=\E_{T_-}[\norm{PX}^2]\le C(k+q^2).
\end{equation}
\end{lemma}

\begin{proof}
Set $q'=q/(q-1)$.  H\"older's inequality and \cref{lem:renyi-comparison} give
\begin{align*}
  \E_{T_-}[\norm{PX}^2]
  &\le
  \left\|\frac{\dd\mu_{T_-}}{\dd\mu_0}\right\|_{L^{q'}(\mu_0)}
  \left(\E_0[\norm{PX}^{2q}]\right)^{1/q}\\
  &=\exp\left(\frac{\mathrm{D}_{q'}(\mu_{T_-}\|\mu_0)}q\right)
  \left(\E_0[\norm{PX}^{2q}]\right)^{1/q}\\
  &\le e\left(\E_0[\norm{PX}^{2q}]\right)^{1/q}.
\end{align*}
Applying \cref{thm:paouris} with $p=2q$ bounds the last line by
$C(\sqrt{k}+q)^2\le C(k+q^2)$.
\end{proof}

\begin{corollary}[Spectral profile at the earlier time]\label{cor:spectral-profile}
Under the assumptions of \cref{lem:projected-second-moment}, let
\[
  a_1(M_{T_-})\ge\cdots\ge a_n(M_{T_-})
\]
be the eigenvalues of $M_{T_-}$.  Then
\begin{equation}\label{eq:spectral-profile}
  a_j(M_{T_-})\le C\left(1+\frac{q^2}{j}\right),
  \qquad 1\le j\le n.
\end{equation}
Consequently,
\begin{equation}\label{eq:earlier-norms}
  \norm{M_{T_-}}_{\op}\le Cq^2,
  \qquad
  S_{T_-}\le C\sqrt{n+q^4}.
\end{equation}
\end{corollary}

\begin{proof}
Let $P$ project onto the span of the top $j$ eigenvectors of $M_{T_-}$.  Ky Fan's principle and
\cref{lem:projected-second-moment} yield
\[
  ja_j(M_{T_-})
  \le\sum_{i=1}^j a_i(M_{T_-})
  =\Tr(PM_{T_-})
  \le C(j+q^2),
\]
which proves \eqref{eq:spectral-profile}.

Taking $j=1$ gives the operator-norm bound.  Squaring \eqref{eq:spectral-profile} and summing gives
\[
  S_{T_-}^2
  \le C\sum_{j=1}^n\left(1+\frac{q^2}{j}\right)^2
  \le C\left(n+q^4\sum_{j=1}^{\infty}j^{-2}\right)
  \le C(n+q^4).
\]
\end{proof}

\subsection{Short-time propagation}\label{sec:upper-propagation}

The following stability statement depends only on the differential estimates for
$\norm{M_t}_{\op}$ and $S_t$.

\begin{lemma}[Short-time stability]\label{lem:short-time-stability}
Suppose that $0\le a<b$ and that $\alpha,\beta>0$ satisfy
\begin{equation}\label{eq:stability-initial}
  \norm{M_a}_{\op}\le\alpha,
  \qquad
  S_a\le\beta.
\end{equation}
If $20\alpha(b-a)<\log2$ and $20\beta(b-a)<\log2$, then, for every $s\in[a,b]$,
\begin{equation}\label{eq:stability-conclusion}
  \norm{M_s}_{\op}\le2\alpha,
  \qquad
  S_s\le2\beta.
\end{equation}
\end{lemma}

\begin{proof}
Let
\[
  \sigma:=\inf\{s\in[a,b]:\norm{M_s}_{\op}=2\alpha\ \text{or}\ S_s=2\beta\},
\]
with $\inf\varnothing=+\infty$.  If $\sigma=+\infty$, the conclusion follows by continuity.

Otherwise, $\sigma\in(a,b]$, and $\norm{M_s}_{\op}\le2\alpha$ and $S_s\le2\beta$ on
$[a,\sigma]$.  By \cref{cor:integrated-differential},
\[
  \log\frac{S_\sigma}{S_a}
  \le20\alpha(\sigma-a)
  <\log2
\]
and
\[
  \log\frac{\norm{M_\sigma}_{\op}}{\norm{M_a}_{\op}}
  \le20\beta(\sigma-a)
  <\log2.
\]
Either equality at $\sigma$ would make the corresponding logarithmic ratio at least $\log2$,
because \eqref{eq:stability-initial} holds.  This is a contradiction.  Thus no exit occurs,
proving \eqref{eq:stability-conclusion}.
\end{proof}

\begin{lemma}[Endpoint estimate]\label{lem:endpoint-estimate}
There are universal constants $\varepsilon_*,C_*,C_{\mathrm{sp}}>0$ with the following property.
If $T>0$ and $q$ satisfy
\begin{equation}\label{eq:endpoint-hypotheses}
  q\ge C_{\mathrm{sp}}\max\{1,H(T),T\sqrt n\},
  \qquad
  q^4\le n,
  \qquad
  Tq\le\varepsilon_*,
\end{equation}
then
\begin{equation}\label{eq:endpoint-conclusion}
  \norm{M_T}_{\op}\le C_*q^2,
  \qquad
  S_T\le C_*\sqrt n.
\end{equation}
\end{lemma}

\begin{proof}
Set $T_-=T(1-1/q)$.  Choose $C_{\mathrm{sp}}\ge2$.  Then
\cref{cor:spectral-profile} and $q^4\le n$ give, for a universal $C_0$,
\[
  \norm{M_{T_-}}_{\op}\le C_0q^2,
  \qquad
  S_{T_-}\le2C_0\sqrt n.
\]
The remaining interval has length $T-T_-=T/q$.

Apply \cref{lem:short-time-stability} with
\[
  \alpha=C_0q^2,
  \qquad
  \beta=2C_0\sqrt n,
  \qquad
  a=T_-,
  \qquad
  b=T.
\]
The first smallness condition follows from
\[
  \alpha(b-a)=C_0Tq\le C_0\varepsilon_*.
\]
The second follows from
\[
  \beta(b-a)
  =2C_0\sqrt n\,\frac Tq
  \le\frac{2C_0}{C_{\mathrm{sp}}}.
\]
Choose $\varepsilon_*$ small enough that $20C_0\varepsilon_*<\log2$, and
$C_{\mathrm{sp}}$ large enough that $40C_0/C_{\mathrm{sp}}<\log2$.  Thus the smallness
assumptions of \cref{lem:short-time-stability} hold, and its conclusion proves
\eqref{eq:endpoint-conclusion}, after enlarging $C_*$.
\end{proof}

\subsection{Entropy control}\label{sec:upper-entropy}

The parameter $q$ in \cref{lem:endpoint-estimate} depends on $H(T)$, while $H'(T)$ is controlled by
$S_T^2$.  A first-contact argument closes this feedback loop.

\begin{lemma}[Entropy barrier]\label{lem:entropy-barrier}
There are universal constants $c,C_H>0$ such that, for all sufficiently large $n$ and every
$0\le t\le c n^{-3/8}$,
\begin{equation}\label{eq:entropy-barrier}
  H(t)\le C_Hnt^2.
\end{equation}
\end{lemma}

\begin{proof}
Since $\mu_0$ is isotropic, \cref{thm:letwin-quadratic} with $A=I_n$ gives
$\Var_0(\norm{X}^2)\le8n$.  By \eqref{eq:entropy-derivative} and continuity of the variance,
\begin{equation}\label{eq:entropy-initial}
  \lim_{t\downarrow0}\frac{H(t)}{nt^2}
  =\frac{\Var_0(\norm{X}^2)}{2n}
  \le4.
\end{equation}

Fix the constants in \cref{lem:endpoint-estimate}.  Choose
$\kappa>\max\{5,5C_*^2\}$ and set $F(t):=H(t)-\kappa nt^2$.  By
\eqref{eq:entropy-initial}, $F$ is negative on a punctured right neighborhood of zero.

Suppose that $F$ first vanishes at some $T\in(0,c n^{-3/8}]$.  Then
$H(T)=\kappa nT^2$ and $F'(T)\ge0$.  Set
\begin{equation}\label{eq:q-barrier}
  q:=C_{\mathrm{sp}}\max\{1,H(T),T\sqrt n\}.
\end{equation}
We verify that $q^4\le n$ and $Tq\le\varepsilon_*$.  Since
$T\le c n^{-3/8}$ and $H(T)=\kappa nT^2$,
\begin{equation}\label{eq:contact-term-bounds}
  H(T)\le\kappa c^2n^{1/4},
  \qquad
  T\sqrt n\le c n^{1/8}.
\end{equation}
Hence
\[
  q\le C_{\mathrm{sp}}
  \left(1+\kappa c^2n^{1/4}+c n^{1/8}\right).
\]
Choose $c$ small enough and then $n$ large enough that the right-hand side is at most
$n^{1/4}$.  This proves $q^4\le n$.

Moreover,
\begin{align*}
  Tq
  &\le C_{\mathrm{sp}}\left(T+\kappa nT^3+T^2\sqrt n\right)\\
  &\le C_{\mathrm{sp}}\left(
    c n^{-3/8}
    +\kappa c^3n^{-1/8}
    +c^2n^{-1/4}
  \right).
\end{align*}
The right-hand side is at most $\varepsilon_*$ after increasing the lower bound on $n$.

These bounds verify the remaining hypotheses of \cref{lem:endpoint-estimate}, so
$S_T\le C_*\sqrt n$.
By \eqref{eq:entropy-derivative} and \eqref{eq:var-energy},
\begin{equation}\label{eq:H-derivative-contact}
  H'(T)
  =T\Var_T(\norm{X}^2)
  \le10TS_T^2
  \le10C_*^2nT.
\end{equation}
On the other hand, $F'(T)\ge0$ implies $H'(T)\ge2\kappa nT$.  This contradicts
\eqref{eq:H-derivative-contact} because $\kappa>5C_*^2$.  Thus no first contact exists, and
\eqref{eq:entropy-barrier} follows with $C_H=\kappa$.
\end{proof}

\begin{proposition}[Small-precision covariance bound]\label{prop:small-t-upper}
There is a universal constant $c>0$ such that, for all sufficiently large $n$ and all
$0\le t\le c n^{-3/8}$,
\begin{equation}\label{eq:small-t-upper}
  \norm{\Sigma_t}_{\op}\le\norm{M_t}_{\op}\le C(1+n^2t^4).
\end{equation}
\end{proposition}

\begin{proof}
Let $c$ be no larger than the constant in \cref{lem:entropy-barrier}, decreasing it below if
necessary.  Then that lemma applies throughout $0\le t\le c n^{-3/8}$.
At $t=0$, isotropy gives $\Sigma_0=M_0=I_n$, so the claim is immediate.  Fix
$0<t\le c n^{-3/8}$ and choose
\[
  q:=C_1(1+nt^2),
\]
where $C_1\ge C_{\mathrm{sp}}\max\{1,C_H\}$ is a sufficiently large universal constant.

By \cref{lem:entropy-barrier}, $H(t)\le C_Hnt^2$.  The inequality
$t\sqrt n\le(1+nt^2)/2$ also gives
\[
  q\ge C_{\mathrm{sp}}\max\{1,H(t),t\sqrt n\}.
\]
For $t\le c n^{-3/8}$,
\[
  q\le C_1(1+c^2n^{1/4}),
  \qquad
  tq\le C_1(t+nt^3).
\]
After decreasing $c$ if needed and taking $n$ sufficiently large, these estimates imply
$q^4\le n$ and $tq\le\varepsilon_*$.  Thus \cref{lem:endpoint-estimate} applies.

It follows that
\[
  \norm{M_t}_{\op}\le C_*q^2
  \le C(1+nt^2)^2
  \le C(1+n^2t^4).
\]
Finally, $M_t-\Sigma_t=m_tm_t^{\mathsf T}\succeq0$, so
$\norm{\Sigma_t}_{\op}\le\norm{M_t}_{\op}$.
\end{proof}

\subsection{Completion of the proof}\label{sec:upper-completion}

\begin{proof}[Proof of \cref{thm:upper-general}]
It suffices to prove the result for all sufficiently large $n$.  Indeed,
$\E_t[\norm{X}^2]$ is nonincreasing by \eqref{eq:flow-basic}, and hence
\[
  \norm{\Sigma_t}_{\op}
  \le\Tr(\Sigma_t)
  \le\E_t[\norm{X}^2]
  \le\E_0[\norm{X}^2]
  =n.
\]
Enlarging the universal constant therefore covers the finitely many remaining dimensions.  Let
$c$ be the constant in \cref{prop:small-t-upper}.

Set
\begin{equation}\label{eq:tstar}
  t_*:=n^{-2/5}.
\end{equation}
Since $t_*/(c n^{-3/8})=c^{-1}n^{-1/40}$, we have $t_*\le c n^{-3/8}$ for all sufficiently
large $n$.

If $0\le t\le t_*$, \cref{prop:small-t-upper} gives
\[
  \norm{\Sigma_t}_{\op}
  \le C(1+n^2t_*^4)
  =C(1+n^{2/5})
  \le Cn^{2/5}.
\]
If $t\ge t_*$, the Brascamp--Lieb estimate \eqref{eq:BL} gives
\[
  \norm{\Sigma_t}_{\op}
  \le\frac1{2t}
  \le\frac1{2t_*}
  =\frac12n^{2/5}.
\]
This proves \eqref{eq:upper-logconcave}.
\end{proof}

Taking $\mu_0$ to be uniform on an isotropic convex body yields $Q_n\le Cn^{2/5}$.

\section{The lower bound}\label{sec:lower}

We construct an unconditional isotropic body and a Gaussian tilt with axial variance of order
\(n^{2/5}\).  The raw body couples one distinguished coordinate to the total quadratic energy of
the other coordinates.  Moderate deviations determine its raw axial scale; after diagonal
isotropization, a tilt at the matching precision preserves a constant fraction of every transverse
slice in a Gaussian-scale axial window.

A set \(E\subset\R^m\) is called \emph{unconditional} if, for every
\(x=(x_1,\ldots,x_m)\in E\) and \(\varepsilon=(\varepsilon_1,\ldots,\varepsilon_m)\in\{-1,1\}^m\),
one has \((\varepsilon_1x_1,\ldots,\varepsilon_mx_m)\in E\).  Thus each coordinate sign can be
changed independently.  Throughout this section, we set \(d:=n-1\).  We use raw coordinates
\((x,\lambda)\in\R^d\times\R\) and write \((y,z)\) for the corresponding coordinates after
isotropization.  For positive quantities, we write \(f\asymp g\) when the ratio is bounded above
and below by positive constants independent of \(d\); dependence on a fixed parameter such as
\(s_0\) is allowed.  The symbols \(\lesssim\) and \(\gtrsim\) denote the corresponding one-sided
bounds.

\begin{theorem}[Lower bound for covariance inflation]\label{thm:lower-bound}
There are universal constants \(c>0\) and \(n_0\) such that, for every \(n\ge n_0\), one can choose an
unconditional isotropic convex body \(K_n\subset\R^n\) and a precision \(t_n>0\) satisfying
\[
  \norm{\Cov(\mu_{K_n,t_n})}_{\op}\ge c n^{2/5},
\]
where
\[
  \dd\mu_{K,t}(x)
  \propto
  e^{-t\norm{x}^2}\1_K(x)\,\dd x.
\]
\end{theorem}

In particular, the construction gives \(Q_n\ge c n^{2/5}\) for all sufficiently large \(n\).

\subsection{Construction and slice estimates}
\label{sec:lower-slices}

We begin with the raw body.  The transverse cube is \([-\sqrt3,\sqrt3]^d\); the choice of
\(\sqrt3\) makes the quadratic energy of a uniform point have mean \(d\).  Write
\(\Delta_d:=d^{3/5}\) for the deviation scale and
\(r_d:=\Delta_d^2/d=d^{1/5}\) for the corresponding moderate-deviation rate.

The relations \(\Delta_d/\sqrt d=d^{1/10}\to\infty\) and
\(\Delta_d/d=d^{-2/5}\to0\) place the construction in the moderate-deviation regime.  We use the
following raw body.
\begin{equation}\label{eq:lower-raw-body}
  K_d^{(0)}
  :=
  \left\{
    (x,\lambda)\in[-\sqrt3,\sqrt3]^d\times\R:
    \norm{x}^2+2\Delta_d|\lambda|
    \le d-2\Delta_d
  \right\}.
\end{equation}

The following elementary observation records the geometry and symmetries of the raw body.

\begin{lemma}[Geometry and symmetry of the raw body]\label{lem:lower-geometry}
For all sufficiently large \(d\), \(K_d^{(0)}\) is a compact, full-dimensional, unconditional
convex body.  Its uniform law is centered.  There exist \(a,b>0\) such that the covariance has the
form:
\begin{equation*}
  \Cov\bigl(\operatorname{Unif}(K_d^{(0)})\bigr)
  =\diag(aI_d,b)
\end{equation*}
\end{lemma}

\begin{proof}
The function \((x,\lambda)\mapsto\norm{x}^2+2\Delta_d|\lambda|\) is convex, so its sublevel set
intersected with the cube is convex.  Since
\(d-2\Delta_d>0\) for large \(d\), the origin is an interior point.  The cube bounds \(x\), and
the constraint bounds \(\lambda\), so the set is compact and full-dimensional.  The defining
condition is invariant under every coordinate sign change and under permutations of the
transverse coordinates.  These symmetries imply centering, vanishing cross-covariances, and a
common variance in the transverse coordinates.
\end{proof}

For \(s\ge0\), the transverse slice of \(K_d^{(0)}\) at
\(\lvert\lambda\rvert=s\) is determined by
\(\norm{x}^2\le d-2\Delta_d(1+s)\).  Let
\(V_1,\ldots,V_d\) be independent and uniform on \([-\sqrt3,\sqrt3]\), and write
\(W_i:=V_i^2\).  Then \(\E W_i=1\) and \(\sigma^2:=\Var(W_i)=4/5\).  Since the uniform
measure on the cube is normalized Lebesgue measure, the relative volume of this slice is
\begin{equation}\label{eq:lower-slice-volume}
  G_d(s):=
  \Pp\left\{
    \sum_{i=1}^d(W_i-1)
    \le -2\Delta_d(1+s)
  \right\},
  \qquad s\ge0.
\end{equation}
The Euclidean volume of the slice is therefore \((2\sqrt3)^dG_d(s)\).  It is zero for
\(s>L_d\), where
\(L_d:=d/(2\Delta_d)-1=\frac12d^{2/5}-1\le\frac12d^{2/5}\) for large \(d\).

We next estimate \(G_d\) locally and in the tail; integrating these estimates gives the zeroth
and second axial moments.

For the local estimate, we use the following consequence of Petrov's Cram\'er-series expansion.
We state it explicitly to fix the normalization.

\begin{theorem}[Cram\'er--Petrov moderate deviations {\cite[Chapter~VIII, Theorem~1, p.~218]{Petrov75}}]\label{thm:lower-cramer-petrov}
Let \(Y_1,Y_2,\ldots\) be i.i.d. random variables with \(\E Y_1=0\),
\(\Var(Y_1)=\sigma^2>0\), and \(\E e^{\theta Y_1}<\infty\) for all \(\theta\) in a neighborhood of zero.
Let \(\Phi\) denote the standard normal distribution function.  For every nonnegative sequence
\(u_d=o(d^{1/6})\),
\[
    \sup_{0\le z\le u_d}
  \left|
    \frac{
      \Pp\left\{\sum_{i=1}^d Y_i\le-\sigma\sqrt d\,z\right\}
    }{1-\Phi(z)}
    -1
  \right|
  \longrightarrow 0.
\]
\end{theorem}

\begin{lemma}[Local moderate-deviation estimate]\label{lem:lower-local-moderate}
For every fixed \(s_0>0\), uniformly for \(0\le s\le s_0\),
\begin{equation}\label{eq:lower-local-moderate}
  G_d(s)
  \asymp
  r_d^{-1/2}
  \exp\left\{-\frac{2}{\sigma^2}r_d(1+s)^2\right\}.
\end{equation}
\end{lemma}

\begin{proof}
Set \(Y_i:=W_i-1\).  Then \(\E Y_i=0\), \(\Var(Y_i)=\sigma^2\), and \(Y_i\in[-1,2]\), so the
moment-generating function is finite for every real argument.  For
\(z_{d,s}:=2\Delta_d(1+s)/(\sigma\sqrt d)\), let
\(z_d^\star:=\sup_{0\le s\le s_0}z_{d,s}\).  Since
\(z_d^\star\asymp d^{1/10}=o(d^{1/6})\), \cref{thm:lower-cramer-petrov} gives
\[
  G_d(s)=\bigl(1-\Phi(z_{d,s})\bigr)(1+o(1))
\]
uniformly for \(0\le s\le s_0\).  Writing
\(\varphi(z):=(2\pi)^{-1/2}e^{-z^2/2}\), the standard Gaussian tail bounds give
\[
  \frac{z}{1+z^2}\varphi(z)
  \le 1-\Phi(z)
  \le \frac{\varphi(z)}{z},
  \qquad z>0.
\]
Since \(z_{d,s}\to\infty\) uniformly on \([0,s_0]\), the preceding relation and these inequalities give
\[
  G_d(s)\asymp z_{d,s}^{-1}e^{-z_{d,s}^2/2}
\]
uniformly for \(0\le s\le s_0\).  Since \(z_{d,s}^{-1}\asymp r_d^{-1/2}\) and
\(z_{d,s}^2/2=(2/\sigma^2)r_d(1+s)^2\), uniformly on the same interval,
\eqref{eq:lower-local-moderate} follows.
\end{proof}

The local estimate is uniform only for bounded \(s\).  The following global bound controls the
remaining tail and makes its contribution to the moment integrals negligible.

\begin{lemma}[Global slice tail]\label{lem:lower-slice-tail}
For every \(s\ge0\),
\begin{equation}\label{eq:lower-slice-tail}
  G_d(s)
  \le
  \exp\left\{-\frac89r_d(1+s)^2\right\}.
\end{equation}
\end{lemma}

\begin{proof}
Recall that \(W_i=V_i^2\in[0,3]\).  Hoeffding's inequality for independent variables in intervals
of length \(3\) gives
\(\Pp\{\sum_{i=1}^d(W_i-1)\le-u\}\le\exp\{-2u^2/(9d)\}\) for \(u\ge0\).  Substituting
\(u=2\Delta_d(1+s)\) and using \(\Delta_d^2/d=r_d\) proves the claim.
\end{proof}

The local and global bounds now determine the two axial moments.  The case \(\ell=0\) gives the
total slice mass, while \(\ell=2\) gives its second moment.

\begin{lemma}[Slice moment integrals]\label{lem:lower-slice-integrals}
For \(\ell\in\{0,2\}\) and all sufficiently large \(d\),
\begin{equation}\label{eq:lower-slice-integrals}
  \int_0^\infty s^\ell G_d(s)\,\dd s
  \asymp
  G_d(0)r_d^{-(\ell+1)}.
\end{equation}
\end{lemma}

\begin{proof}
Choose a fixed \(s_0>0\) such that
\(\frac89(1+s_0)^2>\frac2{\sigma^2}\).  For \(0\le s\le s_0\), comparing
\eqref{eq:lower-local-moderate} at \(s\) and at \(0\) gives
\[
  G_d(s)\asymp G_d(0)
  \exp\left\{-\frac2{\sigma^2}r_d(2s+s^2)\right\}.
\]
On this interval, \(2s\le2s+s^2\le(2+s_0)s\).  Hence, for suitable positive constants
\(c_0,C_0\) depending only on \(s_0\),
\(G_d(0)e^{-C_0r_ds}\le G_d(s)\le G_d(0)e^{-c_0r_ds}\).
For large \(d\), \(r_d^{-1}\le s_0\).  Integrating the lower bound over \([0,r_d^{-1}]\) and
the upper bound over \([0,\infty)\) gives the same order, \(G_d(0)r_d^{-(\ell+1)}\).  Thus
\begin{equation*}
  \int_0^{s_0}s^\ell G_d(s)\,\dd s
  \asymp G_d(0)r_d^{-(\ell+1)}.
\end{equation*}

For the remaining range, the global tail bound gives, for \(\ell\in\{0,2\}\),
\begin{equation*}
\begin{aligned}
  \int_{s_0}^\infty s^\ell G_d(s)\,\dd s
  &\le \int_{s_0}^\infty s^\ell e^{-\frac89r_d(1+s)^2}\,\dd s \\
  &=O\left(r_d^{-1}e^{-\frac89r_d(1+s_0)^2}\right).
\end{aligned}
\end{equation*}
The last bound follows by one integration by parts.  The local estimate gives
\(G_d(0)\asymp r_d^{-1/2}e^{-(2/\sigma^2)r_d}\).  Define
\(\eta:=\frac89(1+s_0)^2-\frac2{\sigma^2}>0\).  Then the ratio of the tail integral to
\(G_d(0)r_d^{-(\ell+1)}\) is
\(O(r_d^{\ell+1/2}e^{-\eta r_d})=o(1)\).  Combining the two ranges proves the claim.
\end{proof}

\subsection{Raw covariance and isotropization}\label{sec:lower-isotropization}

The preceding subsection provides the two slice moments needed for the raw covariance.  We first
read off the axial variance from their ratio, then show that the transverse variance stays of
constant order, and finally apply the diagonal scaling that makes the body isotropic.
Let \((X^{(0)},\lambda^{(0)})\sim\operatorname{Unif}(K_d^{(0)})\).  By the sign and
permutation symmetries from \cref{lem:lower-geometry}, this law is centered.  Define
\(a:=\Var(X^{(0)}_1)\) and \(b:=\Var(\lambda^{(0)})\).  Because the law is centered, these are
the corresponding second moments, and \(\Cov\bigl((X^{(0)},\lambda^{(0)})\bigr)=\diag(aI_d,b)\).

The slice ratio gives the narrow axial scale.  It remains to verify that the transverse variance
does not collapse.
\begin{lemma}[Raw covariance scales]\label{lem:lower-raw-scales}
For all sufficiently large \(d\), \(\frac14\le a\le3\).  Moreover,
\begin{equation}\label{eq:lower-raw-axial}
  b
  =
  \frac{\displaystyle\int_0^\infty s^2G_d(s)\,\dd s}
       {\displaystyle\int_0^\infty G_d(s)\,\dd s}
  \asymp r_d^{-2}=d^{-2/5}.
\end{equation}
\end{lemma}

\begin{proof}
For the axial variance, the slices at \(\lambda^{(0)}=s\) and \(\lambda^{(0)}=-s\) have the same
volume.  The two signs contribute a factor \(2\), and the transverse cube contributes
\((2\sqrt3)^d\); both factors occur in the numerator and denominator and cancel.  Then
\cref{lem:lower-slice-integrals} with \(\ell=0\) and \(\ell=2\) gives the displayed estimate.

For the transverse variance, the upper bound is immediate: \(X^{(0)}_1\in[-\sqrt3,\sqrt3]\) and
the law is centered, so \(a=\E[(X^{(0)}_1)^2]\le3\).  For the lower bound, let
\(V=(V_1,\ldots,V_d)\) be uniform on the transverse cube.  For fixed \(x\), write the length of
the admissible \(\lambda\)-interval as \(\mathcal L_d(x):=\bigl(d-2\Delta_d-\norm{x}^2\bigr)_+/\Delta_d\),
where \(u_+:=\max\{u,0\}\).  Fubini's theorem and
\cref{lem:lower-slice-integrals} give
\begin{equation}\label{eq:lower-volume}
  \frac{\operatorname{vol}(K_d^{(0)})}{(2\sqrt3)^d}
  =
  \E\mathcal L_d(V)
  =
  2\int_0^\infty G_d(s)\,\dd s
  \asymp G_d(0)r_d^{-1}.
\end{equation}
Since \(\norm{V}^2=\sum_{i=1}^dW_i\) with \(W_i=V_i^2\in[0,3]\), Hoeffding's inequality gives
\(\Pp\{\norm{V}^2\le d/2\}\le e^{-d/18}\).  Also,
\(\mathcal L_d(x)\le d/\Delta_d=r_d^2\).  Therefore the volume of the portion with
\(\norm{x}^2\le d/2\) satisfies
\begin{equation*}
  \frac{\operatorname{vol}\bigl(\{(x,\lambda)\in K_d^{(0)}:\norm{x}^2\le d/2\}\bigr)}{(2\sqrt3)^d}
  =\E\!\left[\mathcal L_d(V)\1_{\{\norm{V}^2\le d/2\}}\right]
  \le r_d^2 e^{-d/18}.
\end{equation*}
Dividing by \eqref{eq:lower-volume}, the fraction of the whole body in this region is at most
\begin{equation*}
  O\!\left(r_d^{7/2}
  \exp\left\{-\frac d{18}+\frac{2}{\sigma^2}r_d\right\}\right)
  =o(1),
\end{equation*}
where we used \eqref{eq:lower-local-moderate} at \(s=0\); the last relation follows from
\(r_d=d^{1/5}\).  Consequently,
\[
  \E\norm{X^{(0)}}^2\ge \frac d2(1-o(1))\ge\frac d4.
\]
Finally, permutation symmetry gives \(\E\norm{X^{(0)}}^2=da\), so the preceding inequality implies
\(a\ge1/4\) for all sufficiently large \(d\).
\end{proof}

The preceding lemma shows that the transverse standard deviation is \(\sqrt a\asymp1\), whereas
the axial standard deviation is \(\sqrt b\asymp r_d^{-1}\).  Divide each coordinate by its own
standard deviation by setting \(T_d:=\diag(a^{-1/2}I_d,b^{-1/2})\) and
\(K_d:=T_dK_d^{(0)}\).

\begin{corollary}[Diagonal isotropization]\label{cor:lower-isotropization}
The body \(K_d\) is an unconditional isotropic convex body, and \(b^{-1/2}\asymp r_d\).  If
\((x,\lambda)\) denotes raw coordinates and \((y,z)\) the corresponding coordinates in \(K_d\),
then \((y,z)=(a^{-1/2}x,b^{-1/2}\lambda)\), equivalently
\((x,\lambda)=(\sqrt a\,y,\sqrt b\,z)\).
\end{corollary}

\begin{proof}
By \cref{lem:lower-geometry,lem:lower-raw-scales}, the raw covariance is
\(\diag(aI_d,b)\), with \(a\asymp1\) and \(b\asymp r_d^{-2}\).  Hence
\[
  \Cov\bigl(T_d(X^{(0)},\lambda^{(0)})\bigr)
  =T_d\diag(aI_d,b)T_d^{\mathsf T}=I_{d+1}.
\]
The transformed law remains centered, and positive diagonal scaling preserves convexity and all
coordinate sign symmetries.  Thus \(K_d\) is unconditional and isotropic.
\end{proof}

\subsection{The tilted axial marginal}\label{sec:lower-tilted-marginal}

By the preceding subsection, \(K_d\) is isotropic and \(b^{-1/2}\asymp r_d\).  Let \(A\ge1\) be a
constant to be chosen below and set \(t_d:=A r_d^{-2}\).  Then \(d t_d=A\Delta_d\) and
\(\sqrt b\,t_d^{-1/2}\asymp A^{-1/2}\).  The key point is that the transverse tilt lowers the
mean quadratic energy by order \(\Delta_d\).  We use this to control the slice mass throughout
the window \(|z|\le t_d^{-1/2}\), and then integrate the axial marginal.

For \(t>0\), consider the tilted law \(\mu_{K_d,t}\).  Under the change of variables
\(x=\sqrt a\,y\) and \(\lambda=\sqrt b\,z\), its quadratic exponent separates as
\begin{equation}
  t\norm{(y,z)}^2=t\left(\frac{\norm{x}^2}{a}+z^2\right).
\end{equation}
Thus the transverse tilt is a product measure before imposing the slice constraint.  At fixed
\(z\), that constraint is
\begin{equation}\label{eq:lower-slice-constraint}
  \norm{x}^2\le d-2\Delta_d\bigl(1+\sqrt b\,|z|\bigr).
\end{equation}

For \(\tau\ge0\), let \(U_\tau\) have density proportional to
\(e^{-\tau u^2}\) on \([-\sqrt3,\sqrt3]\).  Set \(\tau_d:=t_d/a\).  The bounds \(1/4\le a\le3\) imply
\(t_d/3\le\tau_d\le1\) for all sufficiently large \(d\).  Let
\(U_1,\ldots,U_d\) be independent copies of \(U_{\tau_d}\).  Since
\(\sqrt b\asymp r_d^{-1}\), fix a universal \(B>0\) with \(\sqrt b\le B r_d^{-1}\), and define
\begin{equation}
  p_d(z):=
  \Pp\left\{
    \sum_{i=1}^dU_i^2
    \le d-2\Delta_d\bigl(1+\sqrt b\,|z|\bigr)
  \right\}.
\end{equation}
Thus \(p_d(z)\) is the probability that an independent tilted transverse sample satisfies the
slice constraint at height \(z\).

\begin{lemma}[Uniform tilted slice mass]\label{lem:lower-slice-acceptance}
There is a universal constant \(A_0>0\) such that, whenever \(A\ge A_0\) and \(d\) is sufficiently
large,
\begin{equation}\label{eq:lower-slice-acceptance}
  \inf_{|z|\le t_d^{-1/2}}p_d(z)
  \ge 1-e^{-2r_d/9}
  \ge \frac12.
\end{equation}
\end{lemma}

\begin{proof}
Let \(m(\tau):=\E[U_\tau^2]\).  At \(\tau=0\), \(U_0\) is uniform on
\([-\sqrt3,\sqrt3]\), so \(m(0)=1\).  Differentiation under the integral gives
\(m'(\tau)=-\Var(U_\tau^2)\).  The variance is continuous and positive on \([0,1]\); hence
\(m(\tau)\le1-v_0\tau\) on this interval for a universal \(v_0>0\).  Since
\(\tau_d\ge t_d/3\),
\begin{equation*}
  \E\left[\sum_{i=1}^dU_i^2\right]
  \le d-\frac{v_0}{3}d t_d.
\end{equation*}
Since \(d t_d=A\Delta_d\), the deficit from \(d\) is at least
\(v_0A\Delta_d/3\).

For \(|z|\le t_d^{-1/2}\), the choice of \(B\) gives
\(\sqrt b\,|z|\le B/\sqrt A\).  The distance from the slice threshold to the tilted mean is at least
\[
  \bigl[d-2\Delta_d(1+\sqrt b\,|z|)\bigr]
  -\E\left[\sum_{i=1}^dU_i^2\right]
  \ge
  \left(\frac{v_0A}{3}-2-\frac{2B}{\sqrt A}\right)\Delta_d.
\]
Choose \(A_0\) so that the coefficient is at least \(1\) for every \(A\ge A_0\).  Since
\(U_i^2\in[0,3]\), Hoeffding's inequality for independent variables in intervals of length \(3\)
then yields
\(1-p_d(z)\le\exp\{-2\Delta_d^2/(9d)\}=e^{-2r_d/9}\), uniformly on the stated interval.  This
proves the lemma.
\end{proof}

Fix a universal \(A\ge A_0\).  We now integrate out the transverse coordinates.

\begin{corollary}[Axial variance under the tilt]\label{cor:lower-axial-variance}
Let \(Z\) denote the \(z\)-coordinate under \(\mu_{K_d,t_d}\).  The \(z\)-marginal has density
proportional to \(e^{-t_dz^2}p_d(z)\), and there is a universal constant \(c>0\) such that
\begin{equation}\label{eq:lower-axial-variance}
  \Var_{\mu_{K_d,t_d}}(Z)\ge\frac{c}{t_d}\ge c d^{2/5}.
\end{equation}
\end{corollary}

\begin{proof}
The Jacobian and the transverse normalizing factor are independent of \(z\), which gives the stated
marginal density.  Since \(p_d\) is even, the marginal is even and \(\E Z=0\).  On
\(|z|\le t_d^{-1/2}\), the preceding lemma gives \(p_d(z)\ge1/2\), while \(p_d(z)\le1\) everywhere.
Consequently,
\[
  \Var(Z)
  \ge
  \frac{\frac12\int_{|z|\le t_d^{-1/2}}z^2e^{-t_dz^2}\,\dd z}
       {\int_{\mathbb R}e^{-t_dz^2}\,\dd z}
  =
  \frac{1}{2t_d}
  \frac{\int_{-1}^{1}w^2e^{-w^2}\,\dd w}
       {\int_{\mathbb R}e^{-w^2}\,\dd w}
  \ge\frac{c}{t_d},
\]
where \(w=\sqrt{t_d}\,z\).  The integral ratio is a positive numerical constant.  Since \(A\) is
fixed and \(t_d=A d^{-2/5}\), after adjusting the universal constant this gives the second
inequality in \eqref{eq:lower-axial-variance}.
\end{proof}

\begin{proof}[Proof of \cref{thm:lower-bound}]
For a given \(n\), set \(d=n-1\), \(K_n:=K_d\), and \(t_n:=t_d\).  By
\cref{cor:lower-isotropization}, \(K_n\) is unconditional and isotropic.  Let \(e_n\) denote the
axial unit vector, and let \(X\sim\mu_{K_n,t_n}\).  The covariance matrix is positive semidefinite,
so its operator norm dominates every unit-vector quadratic form.  In particular,
\[
  \norm{\Cov(\mu_{K_n,t_n})}_{\op}
  \ge e_n^{\mathsf T}\Cov(\mu_{K_n,t_n})e_n
  =\Var_{\mu_{K_n,t_n}}(\langle e_n,X\rangle)
  \ge c d^{2/5}
  \ge c\,2^{-2/5} n^{2/5},
\]
where we used \(d=n-1\ge n/2\) for \(n\ge2\).  Absorbing the fixed factor \(2^{-2/5}\) into
the universal constant completes the proof.
\end{proof}

The scale \(2/5\) in this construction results from balancing the two requirements above.  If
\(\Delta=d^\alpha\) and \(r=\Delta^2/d\), then retaining nontrivial slice mass requires
\(dt\gtrsim\Delta\), while a tilt precision producing an axial window of Gaussian scale requires
\(t\gtrsim r^{-2}\).  Balancing the two conditions gives \(\Delta^5\asymp d^3\), hence
\(\Delta\asymp d^{3/5}\), \(t\asymp d^{-2/5}\), and \(t^{-1}\asymp d^{2/5}\).

\section*{AI-disclosure}
We used ChatGPT 5.6 Pro and GPT 5.6 Sol to assist with brainstorming ideas and exploring proof
strategies, and writing proofs. The tool materially affected Sections 2 to 4. Portions of the manuscript text were redrafted or modified with AI assistance across all
sections. The authors verified the correctness and originality of all content including references.

\bibliographystyle{alphaurl}
\bibliography{main.bib}

\end{document}